\pdfoutput=1
\documentclass{amsart}
\usepackage{amsmath}
\usepackage{amssymb}
\usepackage{graphicx}

\theoremstyle{definition}
\newtheorem{definition}{Definition}

\newtheorem{notation}{Notation}
\newtheorem{move}{Move}
\theoremstyle{theorem}
\newtheorem{theorem}{Theorem}
\newtheorem{proposition}{Proposition}

\newtheorem{question}{Question}

\newtheorem{fact}{Fact}

\theoremstyle{remark}

\newcommand{\ri}{\operatorname{RI}}

\begin{document}
\title[Plumbing and computation of crosscap number]{Plumbing and computation of crosscap number}
\author{Noboru Ito}
\address{866 Nakane, Hitachinaka, Ibaraki, 312-8508, Japan}
\email{nito@ibaraki-ct.ac.jp}
\author{Kaito Yamada}
\address{866 Nakane, Hitachinaka, Ibaraki, 312-8508, Japan}
\email{nanigasi.py@gmail.com}
\date{May 30, 2021}
\maketitle
\begin{abstract}
We introduce a ``deformation" of  plumbing. 
We also define a structure of data used in a calculation by computer aid of the crosscap numbers of alternating knots.    
\end{abstract}
\section{Introduction}   
The recent paper \cite{ItoTakimura2018} introduces an unknotting-type number $u^-(K)$ of a knot $K$.   It is known that $u^-(K)$ equals the crosscap number $C(K)$ for every prime alternating knot $K$ \cite{ItoTakimura2019, ItoTakimura2020,  Kindred2020} \footnote{Y.~Takimura introduced an unknotting-type number $u(P)$ for knot projections $P$; seeing $u(P)$, the author NI defined another function $u^-(P)$;   Takimura determined classes of knot projections with $u(P)=1, 2$ or those of $u^-(P)=1, 2$ \cite{ItoTakimura2018}.  The author NI showed $u^-(K)$ $=$ $C(K)$ for any prime alternating knot $K$ in a paper with Takimura \cite{ItoTakimura2020}; T.~Kindred \cite{Kindred2020} independently proved the equality.   Before these works appeared, the related works via band surgery had been given \cite{ItoTakimura2019}.}.  
The following question arises.
\begin{question}\label{realizeQ}
How to construct a spanning surface realizing the crosscap number of an alternating knot   efficiently by a computer program? 
\end{question}
In this paper, we introduce a notion of ``deformed plumbing" for surfaces  (Definition~\ref{DeformedP}) to give an answer to Question~\ref{realizeQ}.  
\begin{theorem}\label{RealizeT}
For any prime alternating knot, a  spanning surface realizing its  crosscap number is given by gluing surfaces induced by connected sums of knots 
and deformed plumbings of surfaces.
\end{theorem}
Theorem~\ref{RealizeT} is an application of a known result of Sakuma \cite{Sakuma1994} (Theorem~\ref{SakumaThm}).  
\begin{theorem}[Sakuma \cite{Sakuma1994}]\label{SakumaThm}
For any special arborescent alternating knot $K$, a spanning surface realizing its crosscap number is given by plumbings of surfaces.  
\end{theorem}
Here,  
\emph{special arborescent knots} were introduced by Sakuma \cite{Sakuma1994}, each of which is the boundary of a spanning surface obtained by plumbing unknotted twisted annuli in a manner specified by the associated weighted tree (e.g. see Figure~\ref{ExtendGraph}).   

\section{A quick review of calculations of crosscap numbers of prime alternating knots}
In this section, we give a quick review of calculations of crosscap numbers of prime  alternating knots. In the rest of this paper, we often use the term ``Reidemeister moves" for knot projections if no confusion arises.  

To begin with, we recall Move~1 and Fact~\ref{propMove1}.  
\begin{move}[\cite{ItoTakimura2018}]
Let $P$ be a knot projection.  
For any pair of simple arcs lying on the boundary of a common region of $P$, the local replacement as in Figure~\ref{Move1} preserving the number of components is called \emph{Move~1}.    
\end{move}
\begin{figure}
\includegraphics[width=10cm]{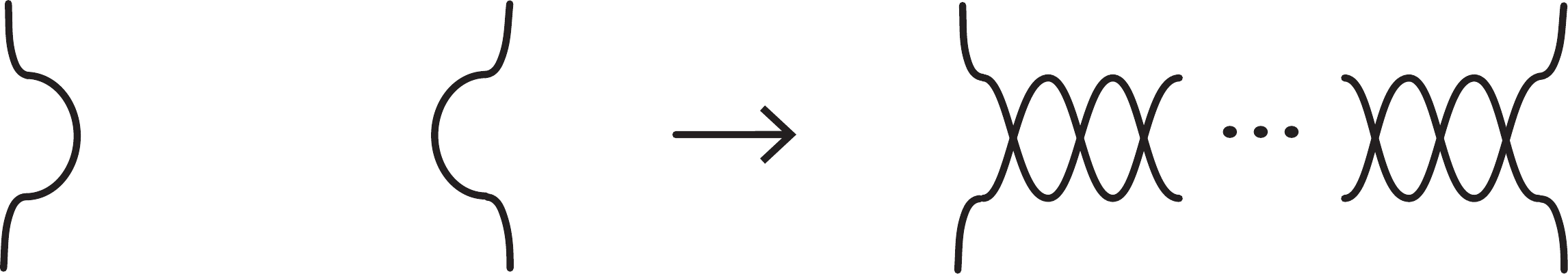}
\caption{A local move as above is called Move1 if it preserves the number of components.}\label{Move1}
\end{figure}
\begin{fact}[{\cite[Proposition~1]{ItoTakimura2018}}]\label{propMove1}
Let $P$ be a knot projection and $O$ the knot projection with no crossings.  
We say that $P \in \langle \{ O \}  \rangle$ if $P$ and $O$ are related by a finite sequence of operations of the first Reidemeister move.  The following conditions are equivalent.  
\begin{enumerate}
\item[(A)] $P$ satisfies $u^- (P) =n$.  
\item[(B)] There exists $Q \in \langle \{ O \} \rangle$ such that $P$ is obtained from $Q$ by applying Move~1 successively $n$ times.   
\end{enumerate}   
\end{fact}
\begin{fact}
Let $K$ be a prime alternating knot and $P$ be a knot  projection induced by an alternating knot diagram of $K$.  Let $C(K)$ be the crosscap number of $K$.  
Then
\[
C(K) = u^-(P).  
\]
\end{fact}

We see the first two steps in the following.  
First, any prime alternating knot with $C(K)=1$ is known as a $(2, n)$-torus knot \cite{Clark1978}.  Second, when we apply Move~1 to a prime alternating knot with $C(K)=1$, there are three possibilities: (A) a $(2m, 2n)$-rational knot, (B): $(2l, 2m-1, 2n-1)$-pretzel knot, and (C) a connected sum of two alternating knots $K$ and $K'$, which satisfy $C(K)$ $=$ $C(K')=1$ (Figure~\ref{StepTwo}).  
\begin{figure}
\includegraphics[width=12cm]{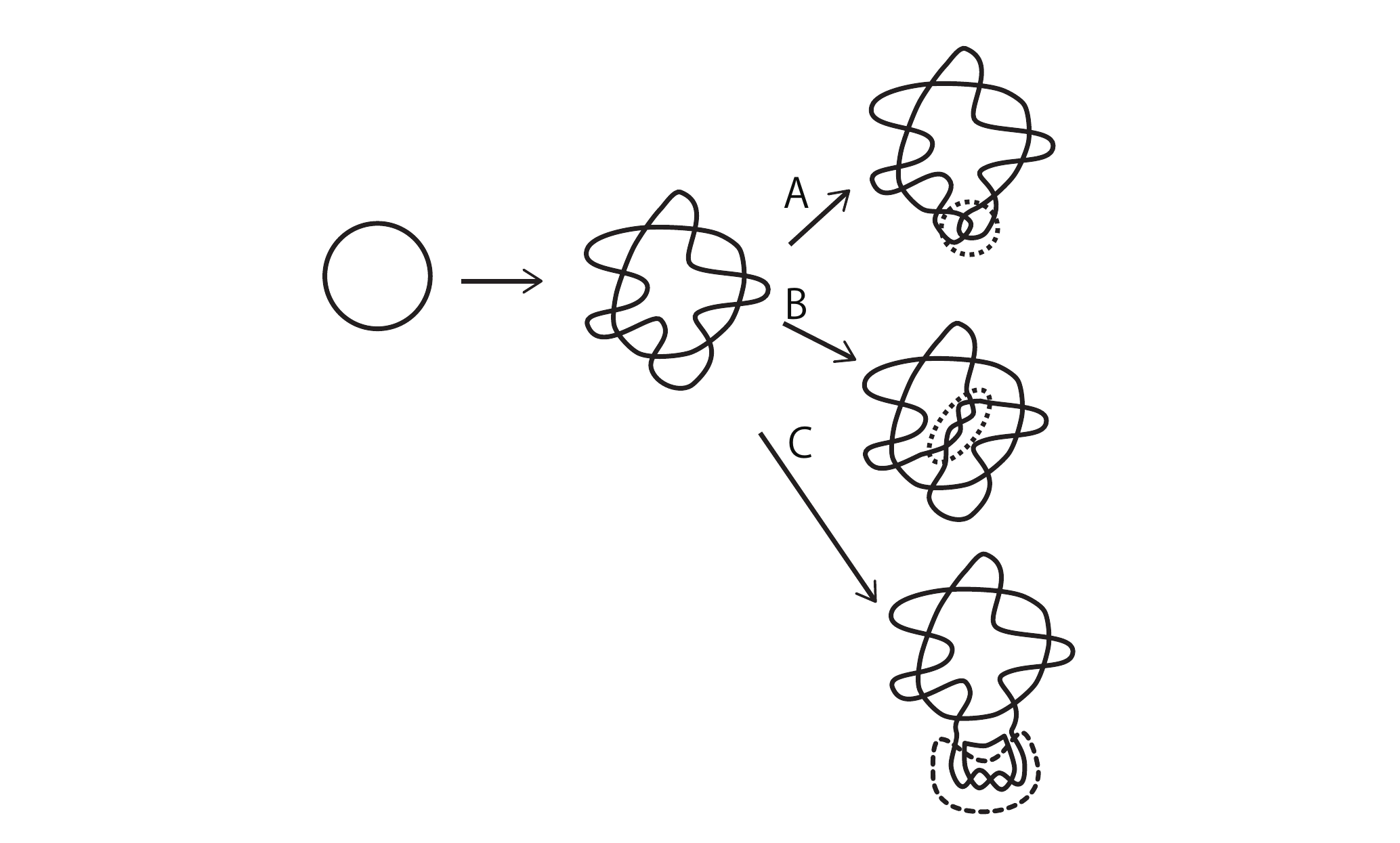}\label{StepTwo}
\caption{
Cases (A), (B), and (C) from the $(2, 2n-1)$-torus knot projection (in this figure, $n=3$ is presented).}\label{figABC}
\end{figure}
Since the case (C) is given by a connected sum, here we will treat cases (A) and (B) in Sections~\ref{PathASec} and \ref{PathBSec} respectively.    

\section{Case~(A)}\label{PathASec}
The case (A) is interpreted as a plumbing as follows. 

Let $\Gamma$ be a planar tree with vertices, each of which is labeled by an integer.  
\begin{figure}
\includegraphics[width=6cm]{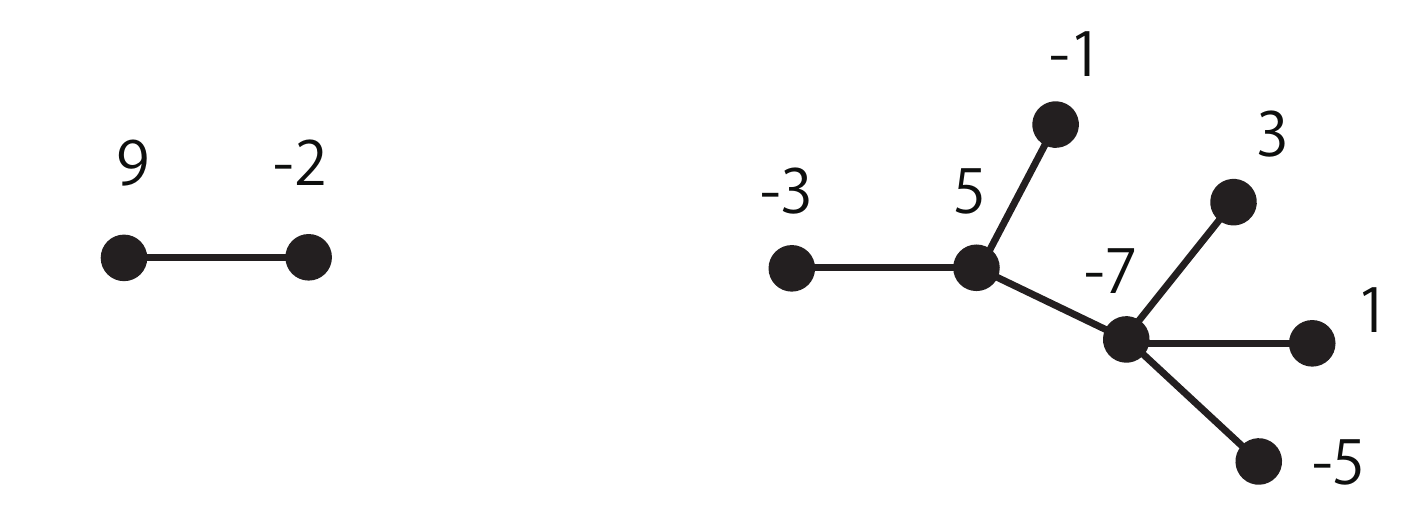}
\caption{A planar tree with integers.}\label{Tree}
\end{figure}
It is a known procedure (\cite{Sakuma1994}), which associates a spanning surface of a knot to the tree $\Gamma$: 
\begin{enumerate}
\item draw a twisted band (see Figure~\ref{Band}) at each vertex $\Gamma$, according to its integer, where each labeled integer corresponds to the signs and the number of half-twists on the bands.   
\item plumb the twisted bands together along the edges of $\Gamma$, as in Figure~\ref{Plumb}. 
\end{enumerate}
\begin{figure}
\includegraphics[width=8cm]{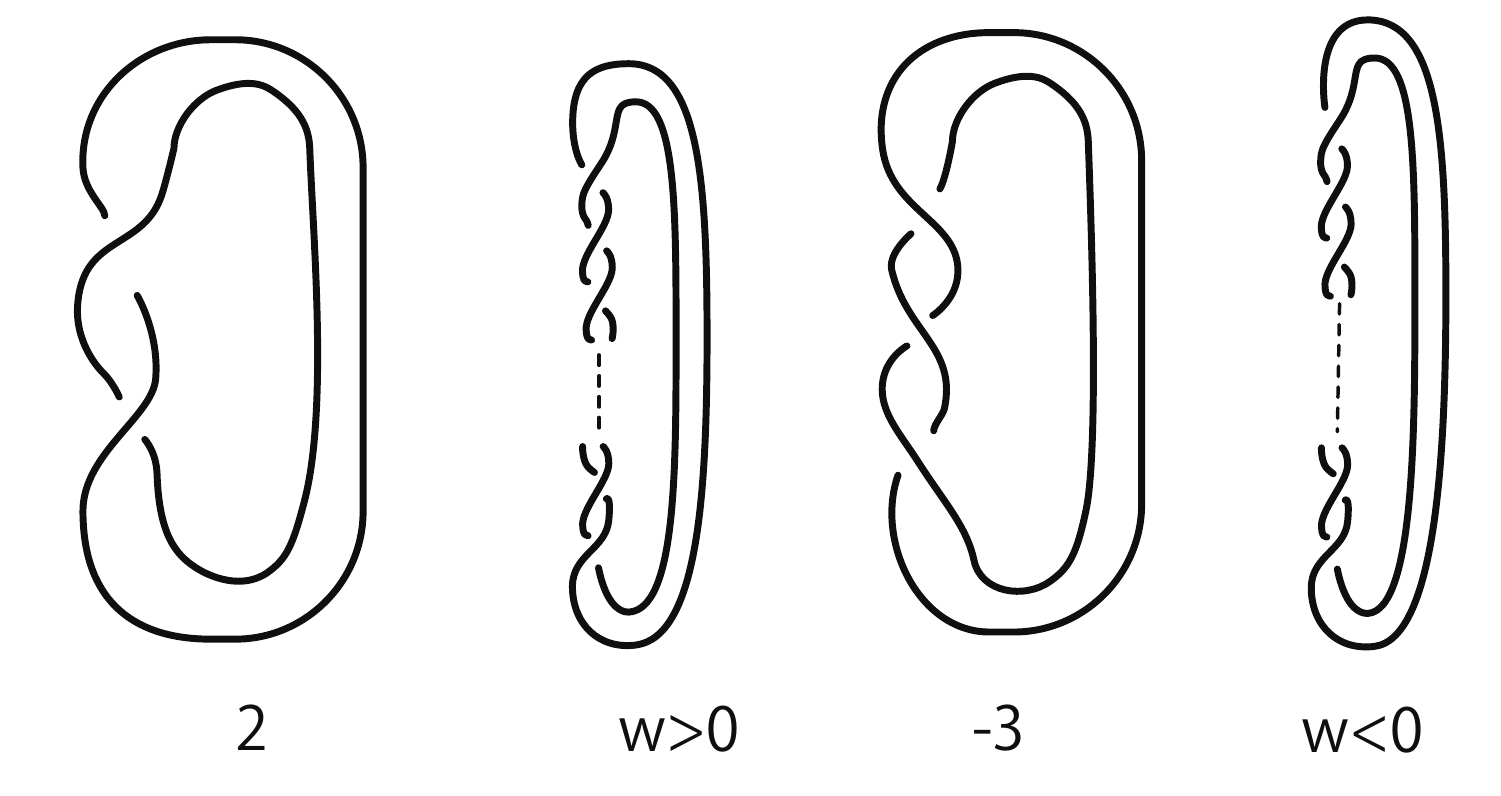}
\caption{A positive or negative twisted band}\label{Band}
\end{figure}
\begin{figure}
\includegraphics[width=12cm]{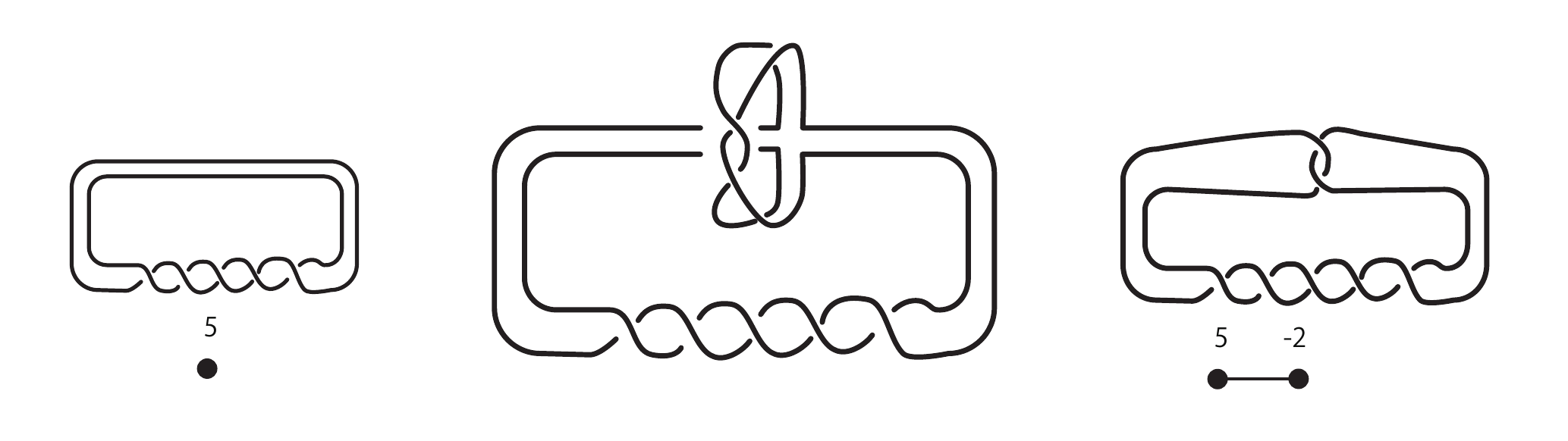}
\caption{Plumbing along an edge of $\Gamma$}\label{Plumb}
\end{figure}
\begin{figure}
\includegraphics[width=12cm]{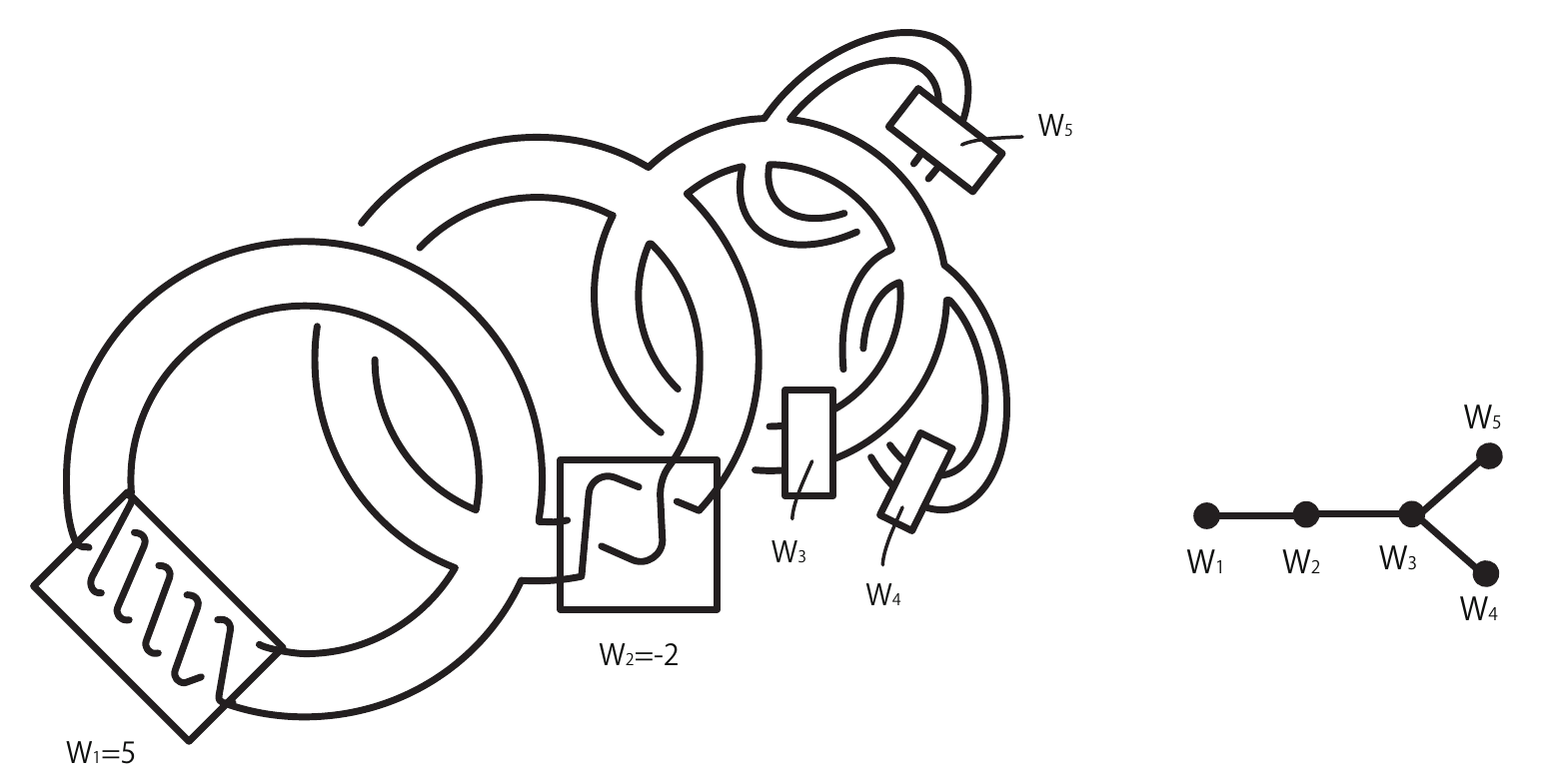}
\caption{A construction of a spanning surface of a special arborescent link \cite{Hirasawa2002}}\label{ExtendGraph}
\end{figure}


\section{Case (B)}\label{PathBSec}
As it is well-known, a plumbing is a special case of Murasugi sum; two surfaces are connected at a square (Figure~\ref{DeformMurasugi}, left).    
Here, we introduce  Definition~\ref{DeformedP}.  
\begin{definition}[Deformed plumbing]\label{DeformedP}
After we apply a plumbing to two spanning surfaces, there is a square in the resulting surface  as in Figure~\ref{DeformMurasugi} (left).  Then we replace the square with a full-twisted band and this deformation is called a \emph{deformation of plumbing}.   The upper right figure in Figure~\ref{DeformMurasugi}  indicates a positive twisting and the lower right figure in Figure~\ref{DeformMurasugi}  indicates a negative twisting.  
Although the induced operation would not arise a plumbing, but it is similar to the original one, thus, we call it a \emph{deformed  plumbing}.    
\end{definition} 
\begin{figure}
\includegraphics[width=14cm]{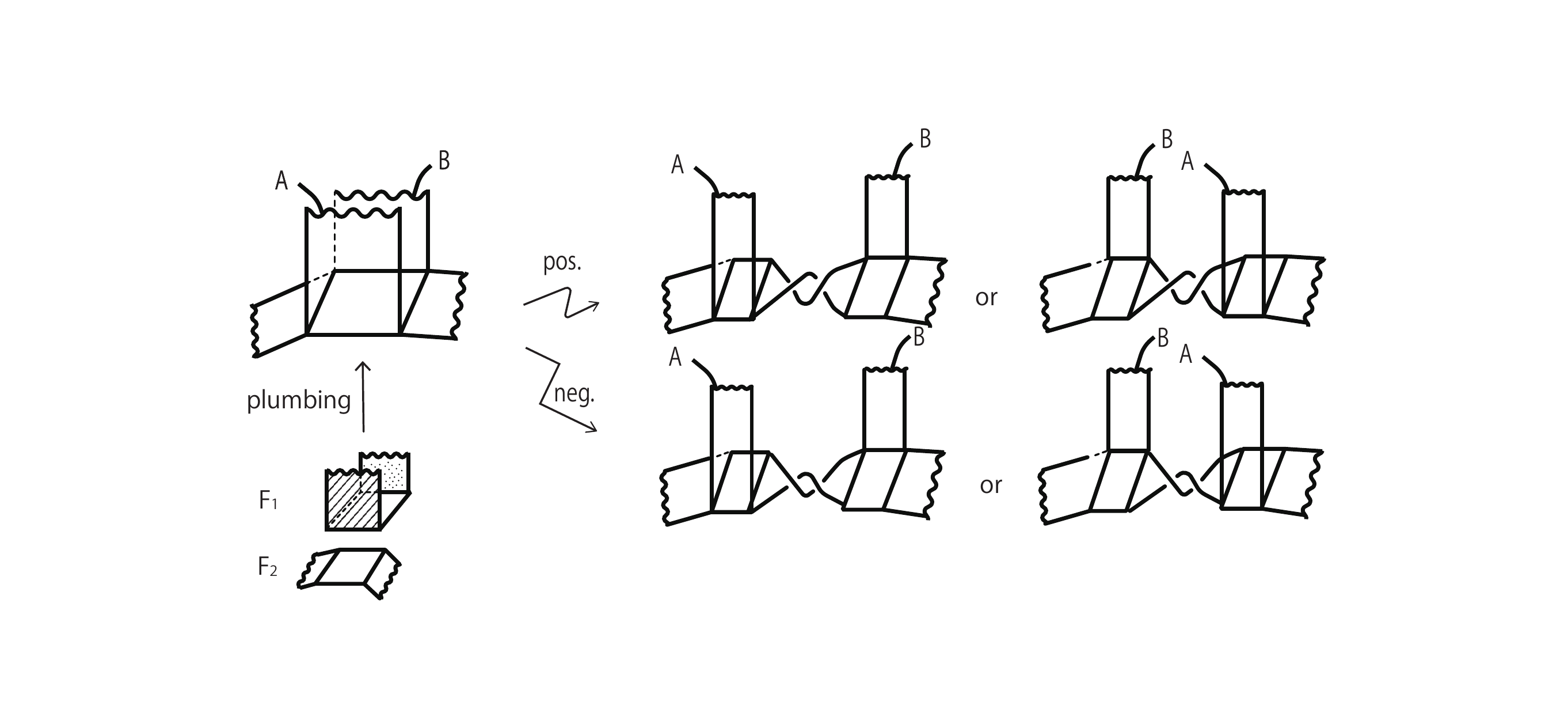}
\caption{A plumbing by identifying squares (left); a positive twisted deformation of plumbing (upper right); a negative twisted deformation of plumbing (lower  right).  }\label{DeformMurasugi}
\end{figure}
For Case (A), by applying a  deformation of a plumbing repeatedly, we have each case (B).  See Figures~\ref{DeformEven} and \ref{DeformOdd}.  
\begin{figure}
\includegraphics[width=12cm]{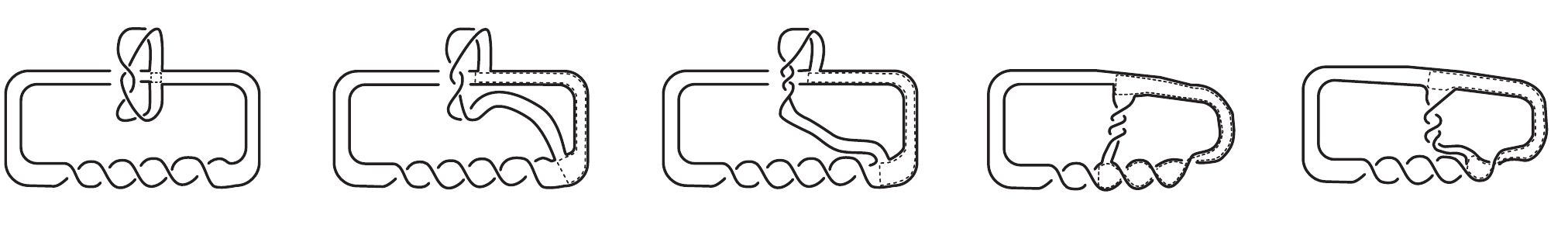}
\caption{Case (A) to Case (B) (Case~1)}\label{DeformEven}
\end{figure}
\begin{figure}
\includegraphics[width=12cm]{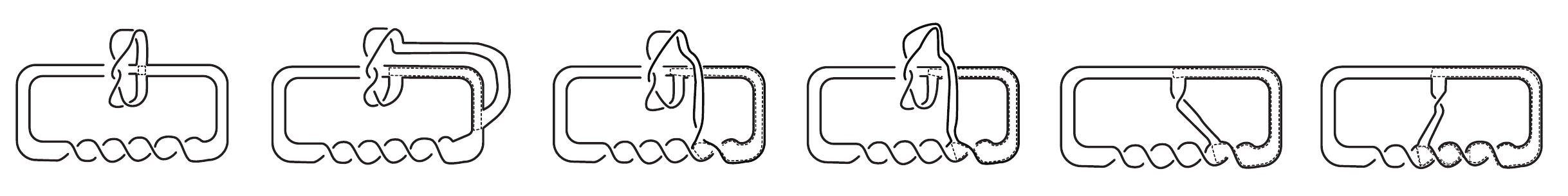}
\caption{Case (A) to Case (B)  (Case~2)}\label{DeformOdd}
\end{figure}

\section{The other cases}
Finally, we consider the other possible cases.  
We recall the construction of special arborescent links by Sakuma \cite{Sakuma1994}.  
\begin{notation}\label{CoreNotation}
For simplifying notations of this construction, 
we only draw cores of the bands (Figure~\ref{Simple}).  Each core will be called a \emph{twisted  band} or a \emph{band} simply.
\end{notation}
\begin{figure}
\includegraphics[width=12cm]{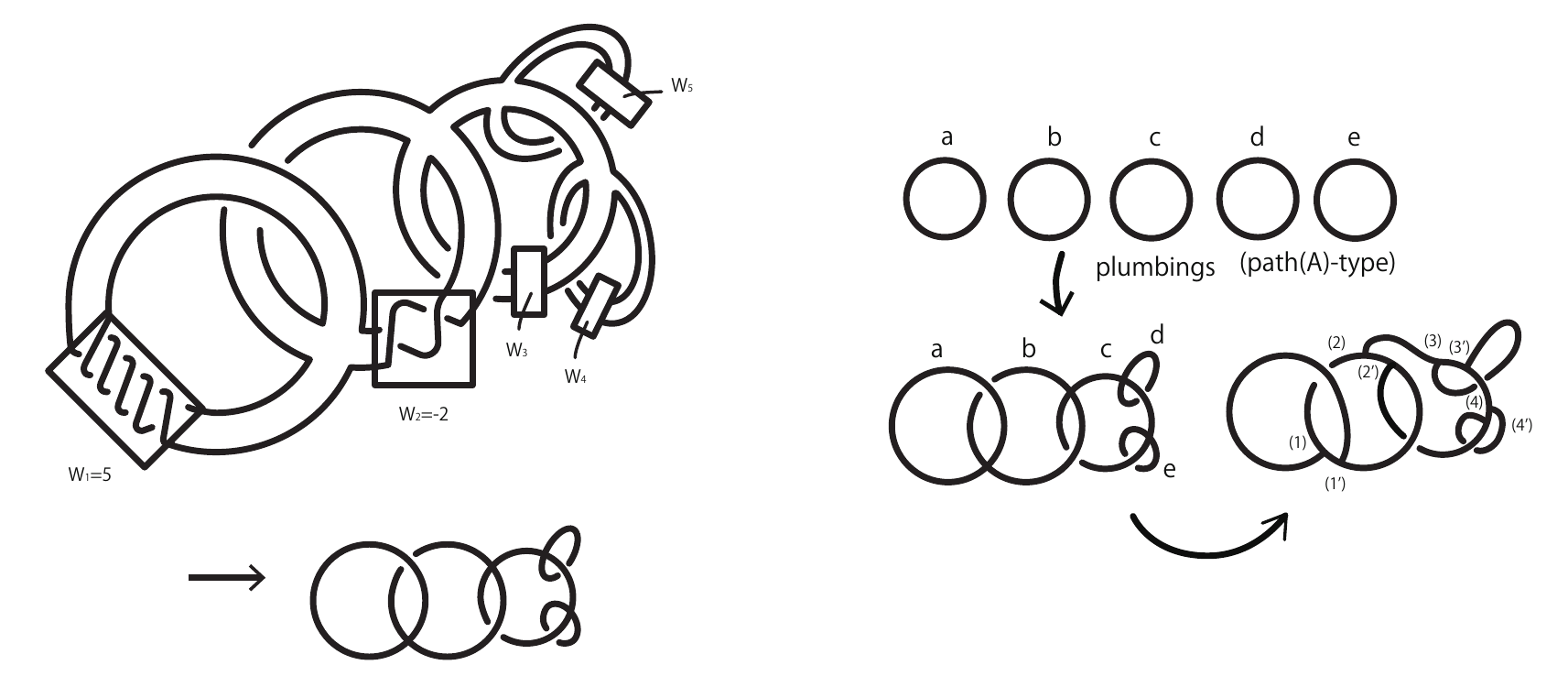}
\caption{Simplifying notations (left), Sliding bands (right)}\label{Simple}
\end{figure}   
By using Notation~\ref{Simple} and apply plumbings and deformed plumbings, we have a graph where endpoints of bands may be slidden as in Figure~\ref{Simple}.
For example, seeing Figure~\ref{Simple} (right), we easily observe that for twisted  bands $a$, $b$, $c$, $d$, and $e$, by applying plumbing, we have a graph as the skeleton of a resulted surface.  When we apply deformations (Definition~\ref{DeformedP}) of plumbings, an endpoint of a single twisted band can be moved along a single band, but the endpoint cannot be moved to a different band.  

Here we prepare one more deformation.  
\begin{definition}[Corner sliding]\label{CornerSliding}
A \emph{corner sliding} is a deformation as in Figure~\ref{CornerSliding}.  
\end{definition}
By definition, every corner sliding does not change the spanning surface. 

\begin{figure}
\includegraphics[width=5cm]{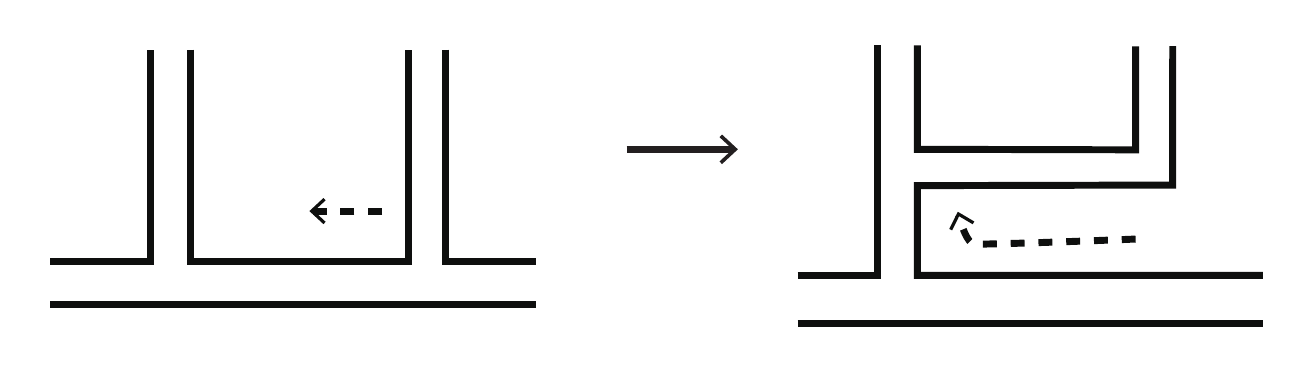}
\caption{Corner sliding}\label{CornerSliding}
\end{figure}
As a result, using plumbings, deformed plumbings and corner sliding, it is possible to attach a twisted band between any of two cores that correspond to vertices of $\Gamma$ as in Section~\ref{ProofTh}.   

\section{Proof of Theorem~\ref{RealizeT}.}\label{ProofTh}
\begin{proof}
Recall Fact~\ref{propMove1}.  Any knot projection is realized by applying Move~1 to a knot projection repeatedly since any knot projection corresponds to  an alternating projection of an alternating knot $K$ with the  crosscap number $C(K)$ ($=u^-(P)=n$ for a positive integer $n$).  Let us compare the Move~1 and deformed plumbings and corner slidings.  
\begin{itemize}
\item Move~1 is applied to any two positions on boundary of a common region.  
\item A deformed plumbing is applied to two positions: a position  is selected arbitrarily and the other can be moved to the corresponding position of Move~1 by twisted deformations and corner slidings.   
\end{itemize}
Clearly, since the above-mentioned two conditions are the same, a single Move~1 one-to-one corresponds to a deformed plumbing.     
\end{proof}

\section{Preliminary for  computation by a computer aid}\label{CAid}

The author KY used a new way to a sufficient information to compute crosscap numbers of prime alternating knots by computer aid.   Our plan is as follows.   
\begin{enumerate}
\item We define a graph $G$   obtained from an alternating knot diagram;  \label{Y1}
\item We define the \emph{reduced} graph for a given $G$; \label{Y2}
\item By inputting the initial data of the case $u^- (D)=0$, we have the alternating knot diagrams with a given crosscap number $n$.   \label{Y3}
\end{enumerate}
In the following, we will provide  (\ref{Y1}) and (\ref{Y2}).

\begin{definition}[Knot Eulerian Graph (e.g.~Figure~\ref{Step2})]
Let $D$ be an alternating knot diagram.   Firstly, we focus on a twisting which will correspond to a vertex.      
Every twisted region of $D$ created by Move~1 is simply presented as in Figure~\ref{Vertex} (upper), and call it a  \emph{real vertex}.
Clearly, every real vertex has two inputs and two outputs.  
We also define an \emph{empty-vertex} \footnote{In this paper, empty vertex has no essential role, but we prepare this notion to realize (\ref{Y3}).} as in an edge (Figure~\ref{Vertex}, lower), and we can add/remove any empty-vertex on an arc.   Every empty-vertex has a single input/output.  Each vertex is equipped with the following information.  
\begin{itemize}
\item Each input corresponds to a unique output.  
\item Each vertex has the valency exactly two or four.  
\item If a vertex corresponds to odd (resp.~even) twisting, we call it an \emph{odd} (resp.~\emph{even}) type.  
\item 
Each twisting with at least two crossings naturally gives a symmetric axis of the real vertex and the left and right sides are labeled by letters $A$ and $B$ in an arbitrary way, wheres an empty-vertex is labeled by ``\emph{None}".   This symmetric axis is called an \emph{axis}.  
If a twisting has exactly one  crossing, we define the axis in an arbitrary way (Figure~\ref{Vertex}).  
\end{itemize}
Then an alternating knot diagram induces all the edges, each of which connects one or two vertex(es).  
Each edge will be labeled by a tuple $(u, v, T_u, T_v)$ of symbols $u, v, T_u$, and $T_v$ defined as follows.   
Each symbol of a $4$ tuple  determines which side a vertex connect to.  
\begin{itemize}
\item $u$ is the initial point of the edge 
\item $v$ is the terminal point of the  edge (the order of $u$ and $v$ reproduces the orientation of the edge.)    
\item $T_u$ is the half part labelled by $A$, $B$, or ``None" having the point $u$.   
\item $T_v$ is the half part labelled by $A$, $B$, or ``None" having the point $v$.
\end{itemize}
By the above construction, an alternating knot diagram induces a graph with even valencies.   
Since by definition, it is an Eulerian graph, we call this graph a \emph{knot Eulerian graph}.  
\end{definition}
\begin{figure}
\includegraphics[width=6cm]{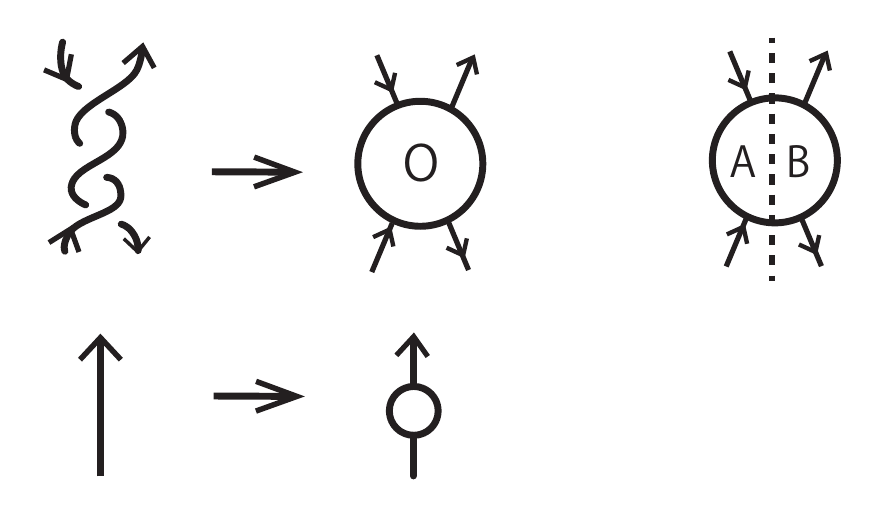}
\caption{Real vertex (upper line); empty-vertex (lower line).  The label ``O" in the vertex  indicates that the number of crossings is an odd number.}\label{Vertex}
\end{figure}
\begin{proposition}
\begin{enumerate}
\item An oriented knot projection gives a knot Eulerian graph.  
\item A knot Eulerian graph gives a family of oriented knot projections of knots, each of which has the same crosscap number. 
\end{enumerate} 
\end{proposition}\label{PropKE}
\begin{proof}

(1) Since, any crossing of oriented alternating knot belongs to a twisting, by replacing each twisting with a real vertex having odd/even-information, we have a $4$-valent graph, which is a knot Eulerian graph.    Here, one may add empty-vertices arbitrarily.  
(2) Since each edge of the knot Eulerian graph has an orientation and each real vertex has information of types (odd/even) and axes, we choose an  oriented knot projection from the knot Eulerian graph (we ignore empty-vertices if necessary).  Here, the information of axes is essential because the twisting direction is determined by an axis (alike Conway notation).  
\end{proof}
By Proposition~\ref{PropKE}, a knot Eulerian graph $G$ gives a knot projection $P$ and then $P$ is called an \emph{underlying knot projection} of $G$.    
\begin{definition}\label{ReKED}
We give a \emph{reduction} algorithm of knot Eulerian graphs as an analogue of reduced alternating knot diagrams in the following. 
\begin{itemize}
\item Step~1: Remove all the empty-vertices.   
\item Step~2 ($\ri^-$ as in Figure~\ref{RI}): Find an edge $(u, u, A, B)$ ($(u, u, B, A)$,~resp.);  remove the vertex $u$ and the edge $(u, u, A, B)$ ($(u, u, B, A)$,~resp.); replace edges $(a, u, T_a, B)$ and $(u, b, A, T_b)$ ($(a, u, T_a, A)$ and $(u, b, B, T_b)$,~resp.) with the edge $(a, b, T_a, T_b)$.  
\item Step~3: Whenever we  unify two vertexes, we do it.  Practically, we need to check it  every pair (as in Figure~\ref{Eg}, upper) of vertices $u, v$ if they can be unified.    
In order to make a computer program, we should list all possible cases of unifying, but here we see only one example.  We believe that the reader can easily recover the full list: 

\noindent If we find two neighboring vertices $P_1$ and $P_2$ as in  Figure~\ref{Eg} (lower left), we remove $P_1$, $P_2$ and add a vertex $Q$ as in Figure~\ref{Eg} (lower right).
\end{itemize}
\end{definition}
\begin{figure}
\includegraphics[width=6cm]{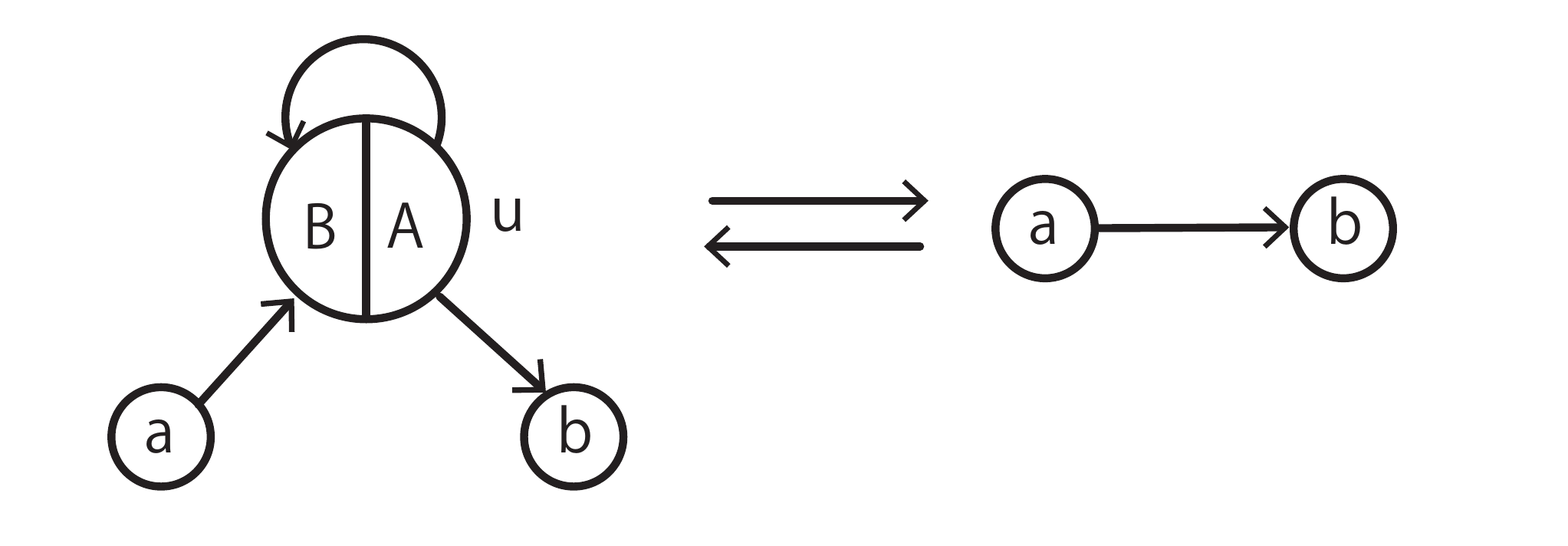}
\caption{$\ri^-$ (the left to the right) and its inverse (the right to the left)}\label{RI}
\end{figure}
\begin{figure}
\includegraphics[width=6cm]{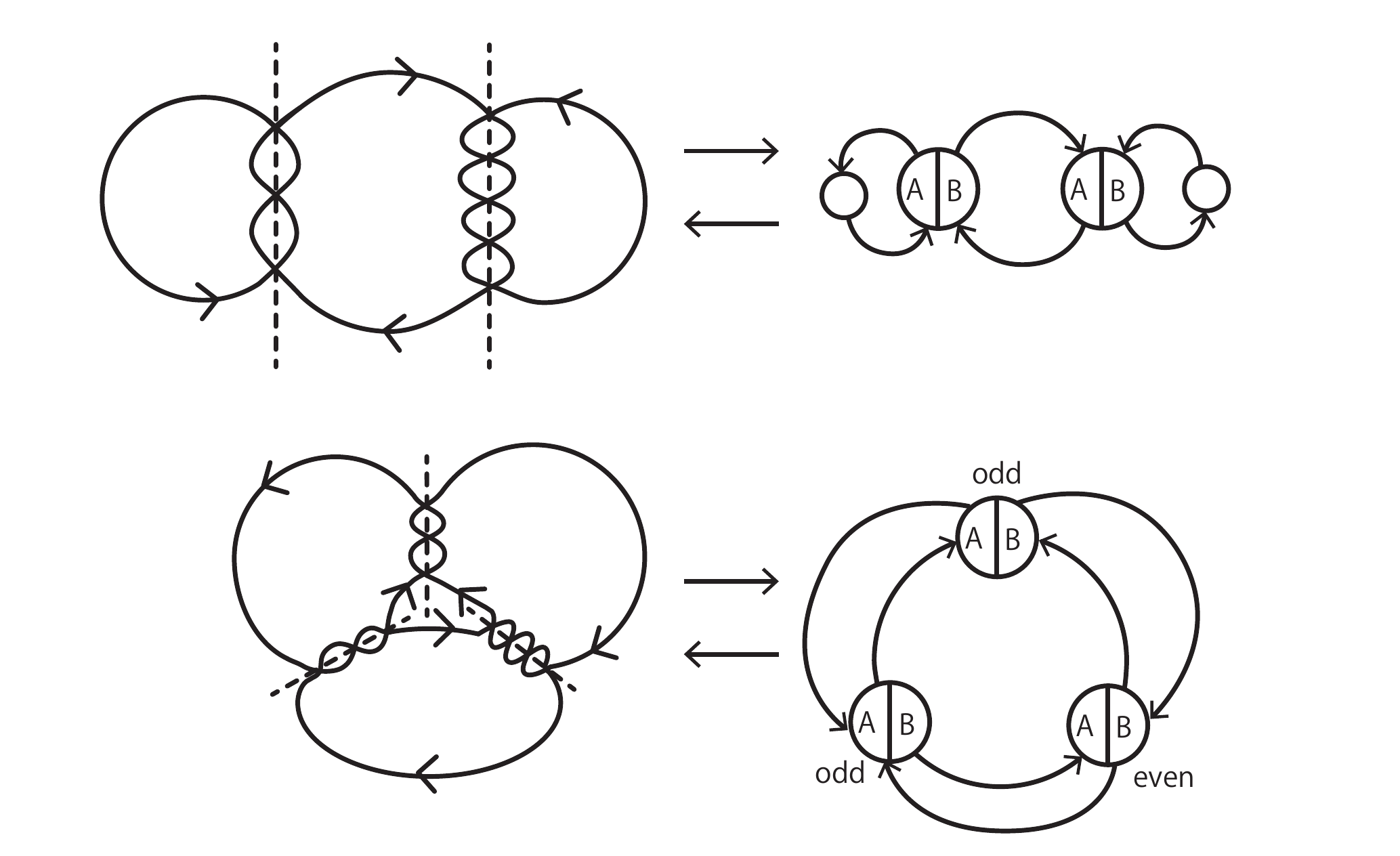}
\caption{Examples of knot Eulerian graphs.}\label{Step2}
\end{figure}
\begin{figure}
\includegraphics[width=4cm]{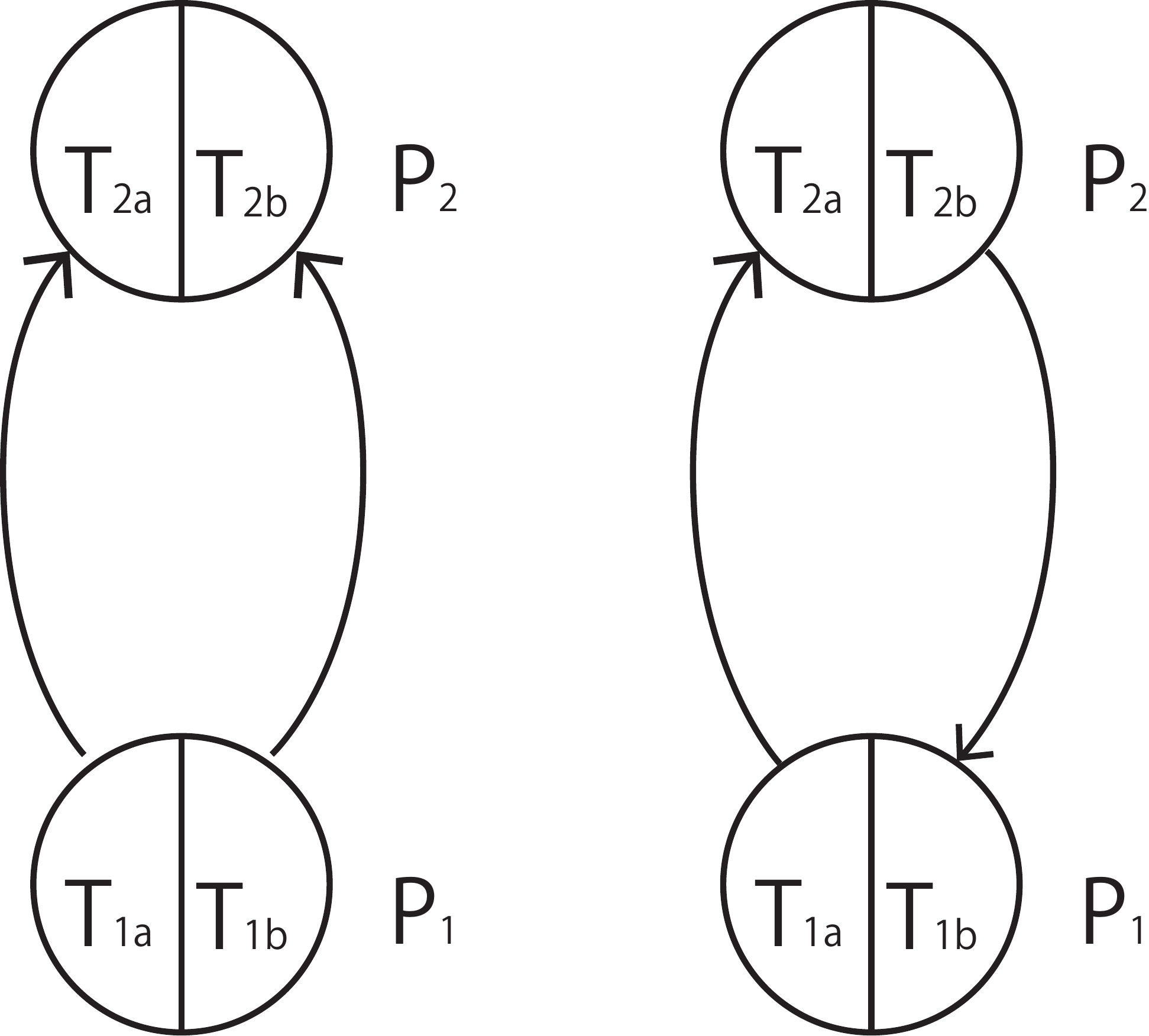}
\includegraphics[width=5cm]{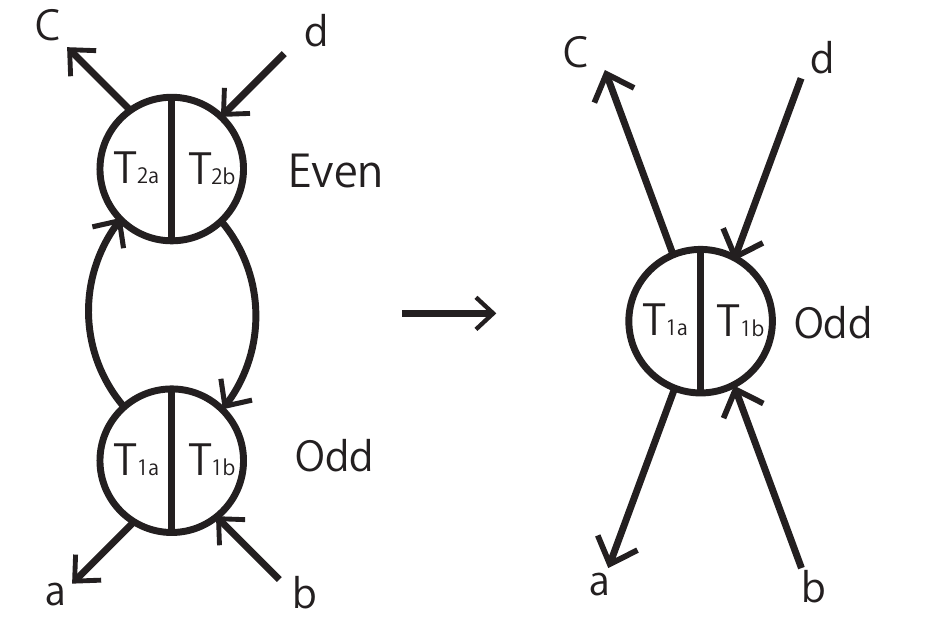}
\caption{}\label{Eg}
\end{figure}
After we complete the reduction for any knot Eulerian graph, the resulting knot projection is called a \emph{reduced} knot Eulerian graph.   
Here, we recall a result by Khovanov \cite{Khovanov1997} (cf.~\cite{ItoTakimura2013} for knot projections on a sphere).    
\begin{fact}[Khovanov \cite{Khovanov1997}]\label{ITLemma}
For a given knot projection $P$, we apply successive first  Reidemeister moves decreasing crossings until we have a knot projection, say $P^{1r}$, with no monogon.   Then, $P^{1r}$ is uniquely determined.   
\end{fact}  
\section{Result of computations}\label{secComp}
The implementation (source code) is given by: 
 
\texttt{https://github.com/nanigasi-san/knot}  

(Please see \texttt{README.md} that includes informations of updates or debugs.)

This realizes (\ref{Y1}) in Section~\ref{CAid} and the judgement of equivalence of any two knot Eulerian graphs.  We have: 
\begin{proposition}
For any knot Eulerian graph, its  reduced knot Eulerian graph is   determined via  Steps~1--3 of Definition~\ref{ReKED} uniquely.  
\end{proposition}
 \begin{proof}
 We will prove the uniqueness.  Let $G$ be a given knot Eulerian graph.  After removing empty vertices in Step~1, suppose that we have $G^{(1)}$.  Then we apply Step~2 to $G^{(1)}$ and have $G^{(2)}$.  Recalling Fact~\ref{ITLemma}, the underlying knot projection corresponding to $G^{(2)}$ is unique.  In order to remove ambiguities decompositions of vertices, we unify any extra vertices belonging to a twisting of an underlying knot projection into a single vertex.   Since by using the definition of digons, it is easy to show that we have the independency of orders of unifications of vertices.  The detail is left to the reader.    
\end{proof}

\section*{Acknowledgements}
The authors would like to thank Professor Mariko Okude for fruitful discussions.   The authors also would like to thank Dr.~ Keita Nakagane for giving us comments for an earlier version of this paper.    

\bibliographystyle{plain}
\bibliography{Ref}
\end{document}